\documentclass[12pt,thmsa]{amsproc}

\usepackage[dvips]{graphics}
\input{amssym.def}
\input{amssym.tex}

\def\gnk{G_{n,k}}
\def\cgnk{\Pi_{n,k}}
\def\cgnj{\Pi_{n,j}}

\def\gnk{G_{n,k}}

\def\bbr{{\Bbb R}}

\def\esp{{\hbox{\rm ess$\,$sup}}}

\def\gnk{G_{n,k}}

\def\rn{\bbr^n}

\def\irn{\intl_{\bbr^n}}

\def\icgr{\intl_{\cgnk}}

\def\part{\partial}
\def\intl{\int\limits}

\def\Gam{\Gamma}

\def\a{\alpha}
\def\om{\omega}

\def\del{\delta}
\def\vp{\varphi}

\def\g{\gamma}
\def\gam{\gamma}

\def\sig{\sigma}
\def\lam{\lambda}
\def\z{\zeta}
\def\e{\varepsilon}
\def\t{\tau}

\newtheorem{theorem}{Theorem}[section]
\newtheorem{lemma}[theorem]{Lemma}

\theoremstyle{definition}

\theoremstyle{remark}

\theoremstyle{corollary}

\numberwithin{equation}{section}

\newcommand{\be}{\begin{equation}}
\newcommand{\ee}{\end{equation}}

\newcommand{\bea}{\begin{eqnarray}}
\newcommand{\eea}{\end{eqnarray}}
\newcommand{\Bea}{\begin{eqnarray*}}
\newcommand{\Eea}{\end{eqnarray*}}

\def\sideremark#1{\ifvmode\leavevmode\fi\vadjust{\vbox to0pt{\vss
 \hbox to 0pt{\hskip\hsize\hskip1em
\vbox{\hsize2cm\tiny\raggedright\pretolerance10000
 \noindent #1\hfill}\hss}\vbox to8pt{\vfil}\vss}}}%

\begin{document}

\title[Weighted norm inequalities for  $k$-plane transforms]
{Weighted norm inequalities for  $k$-plane transforms}

\author{B. Rubin}
\address{
Department of Mathematics, Louisiana State University, Baton Rouge,
LA, 70803 USA}

\email{borisr@math.lsu.edu}





\begin{abstract}
We obtain sharp  inequalities for the k-plane transform, the ``j-plane to k-plane'' transform, and the corresponding dual transforms,
 acting on  $L^p$ spaces with a radial power weight. The operator norms are explicitly  evaluated. Some generalizations and open problems are discussed.
\end{abstract}

\maketitle

\section{Introduction}

\setcounter{equation}{0}

Mapping properties of Radon-like transforms were studied by many authors, e.g.,
\cite{Cal}-\cite{Dr87}, \cite{EO, F,  GSW, Gr, IS, LT, OS,   RT, So,
  Str,  Wo}, to mention a few. Most of the publications deal with  $L^p$-$L^q$
estimates or mixed norm inequalities, when the  problem is to minimize a  gap between necessary and sufficient  conditions  and find the best possible bounds.
 We also mention a series of works devoted to weighted norm estimates
 for Radon-like transforms of radial functions; see, e.g., \cite{DNO, KR12a, KR12b}.

In the present article we show that for the $k$-plane transform in $\bbr^n$ \cite{He, Ru04b}, the more general
``$j$-plane to $k$-plane" transform  \cite[p. 701]{Str},
\cite{Dr86, GK, Ru04d}, and the corresponding dual transforms, sharp estimates can be obtained if the action of
these operators is considered in the $L^p$-$L^q$ setting with $p=q$ and  radial power weights.
In this case the proofs are elementary and self-contained. Our approach is inspired by a series of publications on
operators with homogeneous kernel dating back, probably, to
 Schur \cite{Sch}; see, e.g., \cite {HLP, Sa, Str69}. It is not surprising
  that the same ideas are applicable to some operators of integral geometry, because the  common point is invariance under
rotations and dilations.

   The weighted $L^2$ estimate for the hyperplane Radon transform in $\bbr^n$, $n\ge 3$, was
 obtained by  Quinto \cite[p. 410]{Q} via spherical harmonics and Hardy's
 inequality. This estimate was generalized by Kumar and Ray \cite{KR10}
 to all $n\ge 2$ and all $p$ by making use of the similar spherical harmonics techniques combined with
 interpolation.
 Our approach is completely different, covers these results, and provides  the best possible constants, which
 coincide with explicitly evaluated operator  norms.

Let us proceed to details.
We denote by $\cgnk$   the
 manifold of all non-oriented  $k$-planes $\t$ in
$\rn$; $\gnk$ is the  Grassmann manifold
 of  $k$-dimensional  linear subspaces $\xi$ of $\rn$; $1 \le k\le n-1$.
 Each $k$-plane $\tau$  is parameterized by the pair
$( \xi, u)$, where $\xi \in \gnk$ and $ u \in \xi^\perp$ (the
orthogonal complement of $\xi $ in $\rn$). Thus, $\cgnk$ is a bundle over $\gnk$   with an
$(n-k)$-dimensional fiber.
  The manifold  $\cgnk$ is endowed with the product measure $d\t=d\xi du$,
where $d\xi$ is the
 $O(n)$-invariant probability measure  on $\gnk$ and $du$ denotes the  volume element on $\xi^\perp$.
 We define
 \[
L^p(\cgnk; w)=\{ f: ||f||_{p,w}\equiv ||wf||_p<\infty \},\qquad 1\le p\le \infty,\]
where $||\cdot ||_p$ is the  norm in $L^p(\cgnk)$.
  If   $w(\tau)=|\t|^\nu$, where $|\tau|$
denotes the  Euclidean distance from the plane $\tau \in \cgnk$ to the origin, we also write $L^p(\cgnk; w)=L^p_\nu(\cgnk)$ and
$||f||_{p,w}=||f||_{p,\nu}$.

The $k$-plane transform   of a sufficiently good function $f$ on $\bbr^n$  is a function $R_kf$ on $\cgnk$ defined by
 \be\label{uu7}
 (R_kf)(\t)\equiv (R_kf)(\xi,u) =\intl_{x\in\tau} f(x) \,d_\tau x\equiv \intl_\xi f(y+u) \,d y,
 \ee
where $dy$ is the volume element in $\xi$. The  more general ``$j$-plane to $k$-plane" transform takes a
function $f$ on $\cgnj$ to a function $R_{j,k}f$ on $\cgnk$, $0\le j<k<n$, by the formula
 \be\label{uu7j}
(R_{j,k}f)(\t)\!\equiv \!(R_{j,k}f)(\xi,u)\! = \!\!\intl_{\z\subset\tau}\!\! f(\z) \,d_\tau \z\!\equiv  \!
\!\intl_{\eta\subset\xi} \! \!d_\xi\eta\!\intl_{\eta^\perp\cap \xi} \!\!\! f(\eta, y\!+\!u)\, dy.
\ee
Here $d_\xi\eta$ denotes the  probability measure on the  manifold of all $j$-dimensional linear subspaces
$\eta$ of $\xi$.

\begin{theorem} \label{kk5cds} Let $1\le p \le \infty$,  $\; 1/p+1/p'=1$, $\nu=\mu-k/p'$,  $\mu>k-n/p$.
Then $R_k$ is a linear bounded operator from $L^p_\mu(\bbr^n)$ to $L^p_\nu(\cgnk)$ with the norm
\be\label{3458}
||R_k||=\pi^{k/2}\,\left (\frac{\sig_{n-k-1}}{\sig_{n-1}}\right )^{1/p}\, \frac{\Gam ((\mu +n/p -k)/2)}{\Gam ((\mu +n/p)/2)}
.\ee
\end{theorem}
\begin{theorem} \label{kk5cdj} Let $1\le p \le \infty$,  $\; 1/p+1/p'=1$, $\nu=\mu-(k-j)/p'$,  $\mu>k-n/p-j/p'$.
Then  $R_{j,k}$ is a linear bounded operator from $L^p_\mu(\cgnj)$ to $L^p_\nu(\cgnk)$ with the norm
\be\label{3458j}
||R_{j,k}||=\pi^{(k-j)/2}\,\left (\frac{\sig_{n-k-1}}{\sig_{n-j-1}}\right )^{1/p}\, \frac{\Gam ((\mu +n/p -k+j/p')/2)}{\Gam ((\mu +n/p-j/p)/2)}
.\ee
\end{theorem}
Apart of these theorems, we obtain similar statements for the corresponding dual transforms; see Sections \ref{ss32}, \ref{uu75df}, \ref{ss42}, \ref{ss43}.  The assumptions for $\mu$ and $\nu$ are best possible.
 Generalizations and open problems are discussed in Section 5.

\section{Preliminaries}

\noindent{ \bf Notation.}  In the following $ \; \sigma_{n-1}  = 2 \pi^{n/2}/
 \Gamma(n/2)$ is the area of the unit sphere $S^{n-1}$ in $\rn; \;$ $d\sig(\theta)$ stands for the surface element of $S^{n-1}$;
 $S^{n-1}_+ = \{ x=(x_1, \cdots ,x_n) \in S^{n-1}: \ x_n>0 \}$ is the ``upper" hemisphere
of $S^{n-1}$; $e_1, \ldots, e_n$ are coordinate unit
 vectors; $G=O(n)$ is the  group of orthogonal  transformations of $\bbr^n$ endowed with the
 invariant probability measure. This group acts  on $\gnk$ transitively. For $g,\gam \in G$, we denote
  \[  f_g (x)=f(gx), \qquad \vp_\gam (\xi,u)=\vp(\gam\xi,\gam u).\]

We will need the following simple statements.

\begin{lemma} \label{mmz12} The norm of  a function $\vp(\t)\equiv \vp(\xi,u)\in L^p_\nu(\cgnk)$  can be computed by the formula
\be\label {LLp} ||\vp||_{p,\nu}=\Bigg (\sig_{n-k-1}\intl_0^\infty r^{n-k-1+\nu p} \,dr\intl_{G}
|\vp_\gam (\bbr^k, re_{k+1})|^p\, d\gam \Bigg )^{1/p},\ee
if $\;1\le p<\infty$, and
\be\label {LLpb} ||\vp||_{\infty,\nu}=\underset{\gam \in G, \;r>0}{\esp} \, \{ r^\nu |\vp_\gam (\bbr^k, re_{k+1})|\},\ee
if $\;p=\infty$.
\end{lemma}
\begin{proof} The proof is straightforward and based on the definition
\[||\vp||_{p,\nu}=\left \{\begin{array} {ll}\displaystyle{\Bigg (\,\intl_{\gnk}\!d\xi\intl_{\xi^\perp}
\big | |u|^{\nu} \vp(\xi,u)\big |^p\, du\Bigg)^{1/p}} & \mbox {if $\;1\le p<\infty$,} \\
{}\\
\underset{\xi, u}{\esp} \,\{ |u|^{\nu}|\vp(\xi,u)|\}& \mbox {if $\;p=\infty$.} \\
\end{array}\right.\]
\end{proof}

\begin{lemma}\label {kkooll}
Let $f\in L^1(\bbr^{n-1})$, $\theta=(\theta_1, \cdots, \theta_{n})\in S^{n-1}$. Then
 \be\label{teq2} \intl_{\bbr^{n-1}} f(x)\, dx=\intl_{S^{n-1}_+}f
\Big (\frac{\theta'}{\theta_n}\Big )\,
\frac{d\sig(\theta)}{\theta_n^n}, \qquad \theta'=(\theta_1, \cdots, \theta_{n-1}).\ee
\end{lemma}
\begin{proof}
\bea
 &&\intl_{\bbr^{n-1}} f(x)\, dx=\intl_0^\infty s^{n-2}\, ds\intl_{S^{n-2}}
f (s\om)\, d\sig(\om) \qquad (s=\tan \vp) \nonumber \\
&=&\intl_0^{\pi/2}\frac{\sin ^{n-2} \vp}{\cos^n \vp}\,
d\vp\intl_{S^{n-2}}
 f \Big (\frac{\om\sin\vp}{\cos\, \vp}\Big )\, d\sig(\om)=\intl_{S^{n-1}_+}f
\Big (\frac{\theta'}{\theta_n}\Big )\,
\frac{d\sig(\theta)}{\theta_n^n}.\nonumber\eea
 \end{proof}

 \section{Mapping properties of the $k$-plane transform}\label{mmlaa}

\subsection{Preparations}

The following explicit equalities, reflecting action of $R_k$ on weighted $L^1$ spaces, were obtained in \cite[Theorem 2.3]{Ru04b}.

\begin{lemma} \label{00125} Let
\be\label{uub5} \lam_\mu =\frac{ \Gam ((n-k+\mu)/2) \, \Gam (n/2)}{\Gam ((n+\mu)/2)
\, \Gam ((n-k)/2)}, \qquad \mu>k-n.\ee
Then
\be\label{hhyt} \icgr (R_kf)(\tau) |\t|^{\mu} d\t=\lam_\mu \irn f(x) |x|^{\mu} dx, \ee
\be\label{hhyt1}  \icgr \!(R_kf)(\tau)
\frac{|\t|^{\mu}\;  d\t}{(1+|\t|^2)^{(n+\mu)/2}}\!=\!\lam_\mu \irn\!
f(x)\frac{|x|^{\mu}\; dx}{(1+|x|^2)^{(n-k+\mu)/2}}, \ee provided that
either side of the corresponding equality exists in the Lebesgue
sense.
\end{lemma}

\begin{lemma}\label{gttgzw0}  Let $f\in L^p_\mu(\bbr^n)$, $1\le p \le \infty$, and suppose that
\be\label {78lm9} \mu\,\left \{ \begin{array} {ll} >k-n/p & \mbox{if $\;1< p\le \infty$,}\\
\ge k-n & \mbox{if $\;p=1$.}\\
\end{array}\right.\ee
Then $(R_kf)(\tau)$ is finite for almost all $\tau\in  \cgnk$.
 If (\ref{78lm9}) fails, then there is a function $f_0\in L^p_\mu(\bbr^n)$
 such that $(R_kf)(\tau)\equiv \infty$. For instance,
\be\label{gttgzw1}
f_0 (x)=\frac{|x|^{-\mu}\,(2+|x|)^{-n/p}}{\log^{1/p+\del} (2+|x|)}, \ee
where  $0<\del <1/p'$ if $1< p\le \infty$, and any $\del>0$ if $p=1$.
\end{lemma}
\begin{proof}  If $\mu >k-n/p$,   the first statement follows from (\ref{hhyt1}) by H\'older's inequality.
If  $\mu =k-n$ (for $p=1$),  we set $\tilde \mu =k-n+\e$, $\e>0$. Since
\[\irn\!
\frac{|f(x)|\,|x|^{\tilde \mu}\; dx}{(1\!+\!|x|^2)^{(n-k+\tilde \mu)/2}}\!=\! \irn\!
\frac{|f(x)|}{|x|^{n-k} }\,  \frac{|x|^{\e}\; dx}{(1\!+\!|x|^2)^{\e/2}}\!\le\! \irn\!
\frac{|f(x)|}{|x|^{n-k} }\,dx\!<\!\infty,\]
then  (\ref{hhyt1}) holds with the new parameter $\tilde \mu$ and, therefore,  $(R_kf)(\tau)$ is finite for almost all $\tau\in  \cgnk$.
 The  second statement of the lemma follows from the
Abel type representation  \cite[p. 98]{Ru04b}:
\be\label{990a}(R_kf)(\tau)\!=\!
\sig_{k-1}\intl_{|\tau|}^\infty f_0(r)(r^2 -|\tau|^2)^{k/2 -1} r \,dr, \qquad f(x)\!\equiv \!f_0 (|x|). \ee
\end{proof}

The  scaling argument (cf. \cite[p. 118]{St}) yields the following.
\begin{lemma} \label{pplkj0}Let $1\le p \le \infty$, $1/p+1/p'=1$.  If
$||R_kf||_{p,\nu} \le c \,||f||_{p,\mu}$
for a nonnegative function $f\not\equiv 0$ and a constant $c>0$ independent of $f$, then $\nu=\mu-k/p'$.
\end{lemma}

Lemmas \ref{gttgzw0} and \ref{pplkj0} contain necessary conditions for  the  operator $R_k$ to be  bounded
  from
$L^p_\mu(\bbr^n)$ to $L^p_\nu(\cgnk)$. It will be shown  that these conditions, except  $\mu =k-n$ when $p=1$, are also
sufficient.

The next statement is obvious.
\begin{lemma}\label{bb67z} Let $1\le p \le \infty$,  $\; 1/p+1/p'=1$. If the operator
 $R_k: L^p_\mu(\bbr^n)\!\to \!L^p_\nu(\cgnk)$ is bounded, then
\be\label{iiaaq} \Big | \intl_{\cgnk}(R_kf)(\tau)\,g(\tau)\,d\tau\Big |\le ||R_k||_{L^p_\mu(\bbr^n) \to L^p_\nu(\cgnk)}\,
 ||f||_{p,\mu}\, ||g||_{p',-\nu}\ee
for any $f\!\in \!L^p_\mu(\bbr^n)$ and  $g\!\in \!L^{p'}_{-\nu}(\cgnk)$.
\end{lemma}

The following ``hemispherical'' representation of the $k$-plane transform plays the crucial role in our consideration; cf. \cite[p. 188]{Na} for $k=n-1$, where it was used for different purposes. We denote
 \[
  \bbr^k= \bbr e_1  \oplus \cdots \oplus \bbr e_k, \qquad \bbr^{k+1}=\bbr^k \oplus \bbr e_{k+1},\]
   \[S^k=\!S^{n-1} \cap \bbr^{k+1}, \quad S^k_+ =\{ \theta =(\theta_1, \cdots ,\theta_{k+1})\in S^{k}:\;
  \theta_{k+1}>0\}.\]

\begin{lemma} \label{ll7zw} Suppose that $u\neq 0$ and let $f_g (x)=f(gx)$, where $g\in G$ satifies  $ g\, \bbr^k = \xi$ and $g\,e_{k+1}= u/|u|$.
 Then
\be \label{0xxsr}(R_kf)(\xi, u)\! = \! (R_k f_g)(\bbr^k, re_{k+1})\! =\! r^k \! \intl_{S^k_+}\! f_g
\Big (\frac{r\theta}{\theta_{k+1}}\Big )\,
\frac{d\sig(\theta)}{\theta_{k+1}^{k+1}}, \quad r\! =\! |u|.\ee
\end{lemma}
\begin{proof}  Changing variables, and using (\ref{teq2}) (with $n=k+1$), we obtain
\bea(R_kf)(\xi, u)&=&\intl_{\bbr^k} f_g(z+re_{k+1})\, dz\!=\!(R_k f_g)(\bbr^k, re_{k+1})\nonumber \\
&=&r^k \intl_{\bbr^k}\! f_g(r(z+e_{k+1}))\, dz
\!=\!r^k \!\!\intl_{S^k_+} \!\!f_g
\Big (\frac{r\theta}{\theta_{k+1}}\Big )\,
\frac{d\sig(\theta)}{\theta_{k+1}^{k+1}}. \nonumber \eea
\end{proof}

\subsection{Proof of Theorem \ref{kk5cds}}\label{ss32}
 Denote by $c_k$ the constant on the right-hand side of (\ref{3458}).

STEP 1. Let us show that  $||R_k||\le c_k$. By Lemmas \ref{mmz12} and \ref{ll7zw},
owing to rotation invariance and Minkowski's inequality for integrals, for $1\le p<\infty$ we have
\bea
&&||R_kf||_{p,\nu}=\!\Bigg (\sig_{n-k-1}\intl_0^\infty r^{n-k-1+\nu p} \,dr\!\intl_{G}
|(R_kf)_\gam (\bbr^k, re_{k+1})|^p\, d\gam \Bigg )^{1/p}\nonumber\\
&&=\Bigg (\sig_{n-k-1}\intl_0^\infty r^{n-k-1+\nu p} \,dr\intl_{G}\Bigg |r^k\intl_{S^k_+}f_\gam \left (\frac{r\theta}
{\theta_{k+1}}\right )
\,\frac{d\sig (\theta)}{\theta_{k+1}^{k+1}}\Bigg |^p \,d\gam\Bigg)^{1/p}\nonumber\\
&&\le\sig_{n-k-1}^{1/p}\intl_{S^k_+}\frac{A(\theta)}{\theta_{k+1}^{k+1}}\, d\sig (\theta). \nonumber\eea
Here, for $\nu=\mu-k/p'$,
\[ A(\theta)=\left (\intl_0^\infty r^{n-1+\mu p}\,dr\intl_{G}\Big |f_\gam \left (\frac{r\theta}
{\theta_{k+1}}\right )
\Big |^p\,d\gam\right )^{1/p}=\frac{\theta_{k+1}^{\mu + n/p}}{\sig^{1/p}_{n-1}}\, ||f||_{p,\mu}. \]
If $p=\infty$, we similarly get
\[||R_kf||_{\infty,\nu}\le  \intl_{S^k_+}\frac{A(\theta)}{\theta_{k+1}^{k+1}}\, d\sig (\theta)\]
with
\[ A(\theta)=\underset{r, \gam}{\esp} \, r^\mu\,\Big|f_\gam \left (\frac{r\theta}
{\theta_{k+1}}\right )
\Big |\le \theta_{k+1}^{\mu}\, ||f||_{\infty,\mu}.\]
This gives $||R_kf||_{p,\nu} \le c_k \,||f||_{p,\mu}$, where
\be\label{nn45b} c_k=\left (\frac{\sig_{n-k-1}}{\sig_{n-1}}\right )^{1/p}
\intl_{S^k_+}\frac{d\sig (\theta)}{\theta_{k+1}^{k+1-\mu-n/p}}, \qquad 1\le p\le\infty.\ee
 The last integral equals
\[\sig_{k-1}\intl_0^1 (1-t^2)^{k/2 -1}\, t^{\mu +n/p -k-1}dt=\frac{\pi^{k/2}\,
\Gam ((\mu +n/p -k)/2)}{\Gam ((\mu +n/p)/2)},\]
as desired.  Thus, $||R_k||\le c_k$.

STEP 2. Let us show that $||R_k||\ge c_k$.
To this end, we transform  (\ref{iiaaq}) by choosing $f$ and $g$
in a proper way.
Suppose that both $f$ and $g$ are nonnegative,   $f(x)\equiv f_0 (|x|)$,
$g(\t)\equiv g_0(|\tau|)$, and denote by $I$ the integral on the left-hand side of (\ref{iiaaq}).
 By Lemma \ref{ll7zw},
 \bea I&=&\sig_{n-k-1}\intl_0^\infty g_0(r)\,r^{n-1}\,dr\intl_{S^k_+}
 f_0 \left (\frac{r}{\theta_{k+1}}\right )\frac{d\sig (\theta)}{\theta_{k+1}^{k+1}}\nonumber\\
&=&\label{llpmv}\sig_{n-k-1}\intl_{S^k_+} \frac{d\sig (\theta)}{\theta_{k+1}^{k+1}}
\intl_0^\infty g_0(r)\,f_0 \left (\frac{r}{\theta_{k+1}}\right )\,r^{n-1}\,dr.\eea
Let $1\le p<\infty$. Then $||f||_{p,\mu}$ and $||g||_{p',-\nu}$ in (\ref{iiaaq}) have the form
\bea||f||_{p,\mu}&=&\Bigg (\sig_{n-1}\intl_0^\infty r^{n-1+\mu p}\,f_0^p (r)\,dr\Bigg )^{1/p}, \nonumber\\
||g||_{p',-\nu}&=&\Bigg (\sig_{n-k-1}\intl_0^\infty r^{n-1-\nu p'}\,g_0^{p'} (r)\,dr\Bigg )^{1/p'}.\nonumber\eea
Choose $g_0$ so that $g_0(r)=r^{\mu p}\, f_0^{p-1}$. Then
\[ ||g||_{p',-\nu}=\tilde c\, ||f||_{p,\mu}^{p-1}, \qquad \tilde c=
\left(\frac{\sig_{n-k-1}}{\sig_{n-1}}\right )^{1/p'}.\]
This equality, together with  (\ref{llpmv}) and (\ref{iiaaq}), yields
\[\sig_{n-k-1}\intl_{S^k_+} \frac{d\sig (\theta)}{\theta_{k+1}^{k+1}}
\intl_0^\infty \!\!f_0^{p-1}(r)\,f_0 \left (\frac{r}{\theta_{k+1}}\right )\,r^{n-1+\mu p}\,dr\!
\le \!\tilde c\,||R_k||\, ||f||_{p,\mu}^{p}.\]
Now we assume $f_0(r)=0$  if $r<1$ and $f_0(r)=r^{-\mu-n/p-\e}$, $\e>0$, if $r>1$.
Then $||f||_{p,\mu}^p =\sig_{n-1}/\e p$
and we have
\[
\sig_{n-k-1}\intl_{S^k_+} \frac{d\sig (\theta)}{\theta_{k+1}^{k+1-\mu-n/p-\e}}\le
\tilde c\,\sig_{n-1}\,||R_k||.\]
Passing to the limit as $\e \to 0$, we obtain
\[ ||R_k||\ge  \left(\frac{\sig_{n-k-1}}{\sig_{n-1}}\right )^{1/p}\intl_{S^k_+}
\frac{d\sig (\theta)}{\theta_{k+1}^{k+1-\mu-n/p}},\]
as desired; cf. (\ref{nn45b}).

 If $p=1$, then $\nu=\mu$. We choose $g_0(r)=r^{\mu}$ and proceed as above.
 If $p=\infty$,   we choose $f_0 (r)=r^{-\mu}$. Then   $ ||f||_{\infty,\mu}=1$ and, by (\ref{iiaaq}),
 \[\intl_{S^k_+} \frac{d\sig (\theta)}{\theta_{k+1}^{k+1-\mu}}
\intl_0^\infty g_0(r)\,r^{n-1-\mu}\,dr
\le ||R_k||\, \intl_0^\infty g_0(r)\,r^{n-1-\nu}\,dr.\]
 Let $g_0(r)=0$  if $r<1$ and $g_0(r)=r^{-\del}$, where $\del$ is big enough. Then
 \[\frac{\nu -n+\del}{\mu -n+\del}\intl_{S^k_+} \frac{d\sig (\theta)}{\theta_{k+1}^{k+1-\mu}}\le \!||R_k||.\]
 Letting $\del \to \infty$, we obtain the result. \hfill $\square$

Some comments are in order.
 In the case $p=1$, the constant (\ref{3458})
coincides with (\ref{uub5}). In the case $p=2$ it differs from that in \cite[p. 410]{Q}.
The case  $p=1$, $\mu =\nu=k-n$, is skipped in Theorem \ref{kk5cds}, though  it is included in Lemma \ref{gttgzw0}.
The reason is that the  boundedness of $R_k$ from $L^1_{k-n}(\bbr^n)$ to $L^1_{k-n}(\cgnk)$ fails to be hold. Take, for instance,
$f(x)=f_0 (|x|)$ with $f_0 (r)\equiv 0$ if $r<10$ and $f_0 (r)=r^{k-\del}$, $\del>0$, otherwise. Clearly, $f\in L^1_{k-n}(\bbr^n)$, however,
 by (\ref{990a}), $$\int_{\cgnk}(R_k f)(\t)|\t|^{k-n}\,d\t=c \intl_0^\infty \frac{dt}{t}\intl_t^\infty f_0(r)(r^2 -|\tau|^2)^{k/2 -1} r \,dr=\infty.$$

\subsection{The dual $k$-plane transform}\label{uu75df}

The dual $k$-plane transform of a function $\vp$ on $\cgnk$ is defined by the formula
\be\label{00k7g}(R_k^*\vp)(x)=\intl_{G}\vp(\g\bbr^k +x)  \,d\g, \quad x \in \rn, \ee
and satisfies the duality relation \cite{He, Ru04b}
\be\label{ionb}\intl_{\rn}(R_k^*\vp)(x)\,f(x)\,dx=\intl_{\cgnk}(R_kf)(\t)\,\vp(\t)\,d\t.\ee
The following lemma gives precise information about the $L^1$ case.

\begin{lemma} \label{hhaw1} \cite[Theorem 2.3]{Ru04b}   Let
\be\label{uub5d} \tilde \lam_\mu =\frac{\pi^{k/2} \, \Gam (-\mu/2)}{\Gam ((k-\mu)/2)},
\qquad \mu<0, \qquad \varkappa \in \{0,1\}.\ee
Then
\be\label{hhyt1d} \irn (R_k^*\vp)(x) \,(\varkappa+|x|^2)^{(\mu-k)/2} \,dx
=\tilde \lam_\mu  \icgr  \vp(\tau)\,  (\varkappa+|\tau|^2)^{\mu/2}\,d\tau,\ee
 provided that
either side of the  equality exists in the Lebesgue
sense.
\end{lemma}

In the $L^p$ case, the scaling argument    yields:
\begin{lemma} \label{pplkjd} Let $1\le p \le \infty$, $1/p+1/p'=1$.
If $||R_k^*\vp||_{p,\nu} \le c \,||\vp||_{p, \mu}$ for a nonnegative
function $\vp\not\equiv 0$ and a constant $c>0$ independent of $\vp$, then $\nu=\mu-k/p$.
\end{lemma}

The following statement is dual to Theorem \ref{kk5cds}.
\begin{theorem} \label{kk5cdsd} Let $1\le p \le \infty$,  $\; 1/p+1/p'=1$, $\nu=\mu-k/p$,
$\mu<(n-k)/p'$. Then the dual $k$-plane transform $R_k^*$ is a linear bounded
operator from $L^p_\mu(\cgnk)$ to $L^p_\nu(\rn)$ with the norm
\be\label{3458d}
||R_k^*||=\pi^{k/2}\,\left (\frac{\sig_{n-k-1}}{\sig_{n-1}}\right )^{1/p'}\,
\frac{\Gam (((n-k)/p' -\mu)/2)}{\Gam ((n/p'+k/p -\mu)/2)}
.\ee
\end{theorem}
\begin{proof} By the duality  (\ref{ionb}),
operators  $R_k: L^p_\mu (\bbr^n)\to L^p_\nu(\cgnk)$  and $ R_k^*: L^{p'}_{-\nu}(\cgnk) \to L^{p'}_{-\mu}(\bbr^n)$
are bounded simultaneously and their norms coincide.
 Hence, replacing $p$ by  $p'$ and making  obvious changes in the statement
 of Theorem \ref{kk5cds},    we obtain the result.
\end{proof}

\section{The $j$-plane to $k$-plane transform}
This transform is defined by (\ref{uu7j}). The reader is referred to \cite{Ru04d} for additional information.

\subsection{Preparations}

 \begin{lemma} \label{00125j} \mbox{\rm (cf. Theorem 2.5 in \cite{Ru04d})}  Let $\,0\le j<k<n$,
\be\label{uub5j} \lam_{j,\mu}=\frac{ \Gam ((n-k+\mu)/2) \, \Gam ((n-j)/2)}{\Gam ((n+\mu-j)/2)
\, \Gam ((n-k)/2)}, \qquad \mu>k-n.\ee
Then
\be\label{hhytj} \intl_{\cgnk} (R_{j,k}f)(\t) \,|\t|^{\mu}\,
d\t =\lam_{j,\mu}  \intl_{\cgnj} f(\z) \,|\z|^{\mu} \,d\z, \ee
\be\label{hhyt1j} \intl_{\cgnk} \frac{(R_{j,k}f)(\t) \,|\t|^{\mu}}{(1+|\t|^2)^{(n+\mu-j)/2}}\,
d\t =\lam_{j,\mu}    \intl_{\cgnj} \frac{f(\z) \,|\z|^{\mu}}{(1+|\z|^2)^{(n+\mu-k)/2}} \,d\z,\ee
provided that
either side of the corresponding equality exists in the Lebesgue
sense.
\end{lemma}

\begin{lemma}\label{gttgzw}  Let $f\in L^p_\mu(\cgnj)$, $1\le p \le \infty$, and suppose that
\be\label {78lm9j} \mu\,\left \{ \begin{array} {ll} >k-n/p -j/p'& \mbox{if $\;1< p\le \infty$,}\\
\ge k-n & \mbox{if $\;p=1$.}\\
\end{array}\right.\ee
Then $(R_{j,k}f)(\t)$ is finite for almost all $\tau\in  \cgnk$.
 If (\ref{78lm9j}) fails, then there is a function $f_0\in  L^p_\mu(\cgnj)$ such
 that $(R_{j,k}f)(\tau)\equiv \infty$. For instance,
\be\label{gttgzw1}
f_0 (\z)=\frac{|\z|^{-\mu}\,(2+|\z|)^{(j-n)/p}}{\log^{1/p+\del} (2+|\z|)}, \ee
where  $0<\del <1/p'$ if $1< p\le \infty$, and any $\del>0$ if $p=1$.
\end{lemma}
\begin{proof} The proof mimics that of Lemma \ref{gttgzw0}, using  H\"older's
inequality in (\ref{hhyt1j}) and the  known formula for radial functions \cite[p. 5051]{Ru04d}:
 \[(R_{j,k}f)(\t)= \sig_{k-j-1}\intl_{|\t|}^\infty  f_1(r)(r^2 -|\t|^2)^{(k-j)/2 -1} r \,dr,
 \quad f(\z)\equiv f_1 (|\z|). \]
 \end{proof}

The  scaling argument (in the fibers) yields the following statement.
\begin{lemma} \label{pplkj}Let $1\le p \le \infty$, $1/p+1/p'=1$.  If
$||R_{j,k}f||_{p,\nu} \le c \,||f||_{p,\mu}$
for a nonnegative function $f\not\equiv 0$ and a constant $c>0$ independent of $f$, then $\nu=\mu-(k-j)/p'$.
\end{lemma}

To obtain an analogue of  (\ref{0xxsr}) for   $R_{j,k}f$, we denote
 \[  \bbr^{j}\!= \!\bbr e_{1}  \oplus \cdots \oplus \bbr e_j, \quad
  \bbr^{k-j}\!= \!\bbr e_{j+1}  \oplus \cdots \oplus \bbr e_k, \quad \bbr^{k-j+1}\!=
  \!\bbr^{k-j}\oplus \bbr e_{k+1},\]
   \[S^{k-j}\!=\!S^{n-1} \cap \bbr^{k-j+1}, \quad S^{k-j}_+ \!=\!\{ \theta \!=\!(\theta_{j+1},
   \cdots ,\theta_{k+1})\!\in\! S^{k-j}:\;
  \theta_{k+1}\!>\!0\}.\]

\begin{lemma} \label{ll7zw0} Let $(R_{j,k}f)(\xi, u)$ be the transformation (\ref{uu7j}) with $u\neq 0$.
If $g\in G$ satisfies
$ g\, \bbr^k = \xi$ and $g\,e_{k+1}= u/|u|$, then
\be (R_{j,k}f)(\xi, u) =r^{k-j}  \intl_{O(k)} d\gam \intl_{S^{k-j}_+}f_{g\gam}
\Big (\bbr^j, \frac{r\theta}{\theta_{k+1}}\Big )\,
\frac{d\sig(\theta)}{\theta_{k+1}^{k-j+1}}, \quad r=|u|.\ee
\end{lemma}
\begin{proof}  Changing variables, we write $(R_{j,k}f)(\xi, u)$  as
\bea
(R_{j,k}f)(g\,\bbr^k,rg\,e_{k+1})
&=&\intl_{G_{k,j}} d\tilde \eta\!\intl_{\tilde \eta^\perp\cap \bbr^k}
f_g(\tilde \eta, re_{k+1}+z)\, dz\nonumber\\
&=&r^{k-j}\intl_{O(k)} d\gam\intl_{\bbr^{k-j}}f_{g\gam}(\bbr^j, r(e_{k+1}+\tilde z))\, d\tilde z.\nonumber\eea
It remains to transform the inner integral using (\ref{teq2}).
\end{proof}

\subsection{Proof of Theorem \ref{kk5cdj}}\label{ss42}

 Denote by $c_{j,k}$ the constant on the right-hand side of (\ref{3458j}).

STEP 1. Let us show that  $||R_{j,k}||\le c_{j,k}$. By Lemmas \ref{mmz12} and \ref{ll7zw0}, for $1\le p<\infty$ we have
\bea
||R_{j,k}f||_{p,\nu}
&=&\Bigg (\sig_{n-k-1}\intl_0^\infty r^{n-k-1+\nu p} \,dr\nonumber\\
&\times&\intl_{G}\Bigg |r^{k-j}\intl_{O(k)}d\gam\intl_{S^{k-j}_+}f_{g\gam}
 \left (\bbr^j,\frac{r\theta}
{\theta_{k+1}}\right )\,\frac{d\sig (\theta)}{\theta_{k+1}^{k-j+1}}\Bigg |^p \,dg\Bigg)^{1/p}\nonumber\\
&\le&\sig_{n-k-1}^{1/p}\intl_{O(k)}d\gam\intl_{S^{k-j}_+}\frac{A_\gam(\theta)}{\theta_{k+1}^{k-j+1}}\, d\sig (\theta), \quad \nu=\mu-(k-j)/p',\nonumber\eea
where
\[ A^p_\gam(\theta)=\intl_0^\infty r^{n-j-1+\mu p}\,dr\intl_{G}\Big |f_{g\gam} \left (\bbr^j,\frac{r\theta}
{\theta_{k+1}}\right )
\Big |^p\,dg=\frac{\theta_{k+1}^{n-j+\mu p}}{\sig_{n-j-1}}\, ||f||^p_{p,\mu}. \]
Hence, $||R_{j,k}f||_{p,\nu} \!\le \!c_{j,k} \,||f||_{p,\mu}$,
\be\label{nn45bj} c_{j,k}\!=\!\left (\frac{\sig_{n-k-1}}{\sig_{n-j-1}}\right )^{1/p}\intl_{S^{k-j}_+}
\frac{d\sig (\theta)}{\theta_{k+1}^{k+1-\mu-n/p-j/p'}}.\ee
This result also covers the case $p=\infty$, when the calculation is straightforward.
The last integral  gives the constant in (\ref{3458j}).

STEP 2. To prove that $||R_{j,k}||\ge c_{j,k}$, we proceed as in the proof of Theorem \ref{kk5cds}
  and use the relevant analogue of  (\ref{iiaaq}).
Let $f$ and $g$ be nonnegative  radial functions,   $f(\z)\equiv f_0 (|\z|)$,  $g(\t)\equiv g_0(|\tau|)$.
 Then
 \[\intl_{\cgnk}\!(R_{j,k}f)(\tau)\,g(\tau)\,d\tau\!=\!\sig_{n-k-1}\intl_{S^{k-j}_+}
 \frac{d\sig (\theta)}{\theta_{k+1}^{k-j+1}}
 \intl_0^\infty \!g_0(r)\,f_0 \left (\frac{r}{\theta_{k+1}}\right )\,r^{n-j-1}\,dr.\]
If $1<p<\infty$, then
\bea ||f||_{p,\mu}&=&\Bigg (\sig_{n-j-1}\intl_0^\infty r^{n-j-1+\mu p}\,f_0^p (r)\,dr\Bigg )^{1/p}, \nonumber\\
||g||_{p',-\nu}&=&\Bigg (\sig_{n-k-1}\intl_0^\infty r^{n-k-1-\nu p'}\,g_0^{p'} (r)\,dr\Bigg )^{1/p'}.\nonumber\eea
Choose $g_0$ so that $g_0(r)=r^{\a}\, f_0^{p/p'}$, $\a=(k-j)/p' +\nu +\mu (p-1)$, and set $f_0(r)=0$  if $r<1$ and
$f_0(r)=r^{-\mu-(n-j)/p-\e}$, $\e>0$, if $r>1$.
Then, as in the proof of Theorem \ref{kk5cds}, we get

\[  \left (\frac{\sig_{n-k-1}}{\sig_{n-j-1}}\right )^{1/p}\intl_{S^{k-j}_+}
\frac{d\sig (\theta)}{\theta_{k+1}^{k+1-\mu-n/p-j/p'-\e}}\le ||R_{j,k}||.\]
It remains to pass to the limit as $\e \to 0$; cf. (\ref{nn45bj}).
The cases $p=1$ and $p=\infty$ are treated as in the proof of Theorem \ref{kk5cds}.
\hfill $\square$

\subsection{The dual $j$-plane to $k$-plane transform}\label{ss43}

For $0\le j<k<n$, the dual  $j$-plane to $k$-plane transform  takes
functions $\vp(\t)\equiv \vp(\xi,u)$ on $\cgnk$ to functions
$(R_{j,k}^*\vp)(\z)\equiv (R_{j,k}^*\vp)(\eta,v)$
on $\cgnj$  by the formula
\be (R_{j,k}^*\vp)(\z)
= \intl_{\t \supset \z} \vp(\t)\, d_\z \t=\intl_{O(n-j)} \vp(g_\eta \rho \bbr^k +v) \, d\rho. \ee
Here $g_\eta \in G$ is an orthogonal transformation that
sends $\bbr^{j}\!= \!\bbr e_{1}  \oplus \cdots \oplus \bbr e_j$ to $\eta$
and $O(n-j)$ is the orthogonal group of  the coordinate plane
$\bbr^{n-j}\!= \!\bbr e_{j+1}  \oplus \cdots \oplus \bbr e_n$.
 This transform
averages  $\vp(\t)$ over all $k$-planes $\t$   containing  the
$j$-plane $\z$. The case $j=0$ gives the dual $k$-plane transform
 (\ref{00k7g}). The duality relation has the form
\be \intl_{\cgnk}(R_{j,k}f)(\tau)\,\vp(\tau)\,d\tau=\intl_{\cgnj}
f(\z) \,(R_{j,k}^*\vp)(\z)\, d\z. \ee

The following exact equalities generalize those in Lemma \ref{hhaw1}.

\begin{lemma} \label{hhaw1j} \cite[Theorem 2.5]{Ru04d}   Let
\[\tilde\lam_{j,\mu} =\frac{\pi^{(k-j)/2} \, \Gam (-\mu/2)}{\Gam ((k-j-\mu)/2)},
 \qquad \mu<0, \qquad \varkappa \in \{0,1\}.\]
Then
\[ \intl_{\cgnj} (R_{j,k}^*\vp)(x) \,(\varkappa +|\z|^2)^{(\mu-k+j)/2} \,d\z
=\tilde\lam_{j,\mu}  \intl_{\cgnk}  \vp(\tau)\,  (\varkappa +|\tau|^2)^{\mu/2}\,d\tau\]
provided that
either side of the  equality exists in the Lebesgue
sense.
\end{lemma}

The dual of Theorem \ref {kk5cdj} is the following.
\begin{theorem} \label{kk5cdsdj} Let $1\le p \le \infty$,  $\; 1/p+1/p'=1$,
$\nu=\mu-(k-j)/p$,  $\mu<(n-k)/p'$. Then  $R_{j,k}^*$ is a linear bounded operator
from $L^p_\mu(\cgnk)$ to $L^p_\nu(\rn)$ with the norm
\be\label{3458d}
||R_{j,k}^*||=\pi^{(k-j)/2}\,\left (\frac{\sig_{n-k-1}}{\sig_{n-j-1}}\right )^{1/p}\,
\frac{\Gam ((\mu +n/p -k+j/p')/2)}{\Gam ((\mu +n/p-j/p)/2)}
.\ee
\end{theorem}
The proof of this statement is
similar to the proof of Theorem \ref{kk5cdsd}.

\section{Some generalizations and open problems}

{\bf 1.} Owing to projective invariance of the Radon transforms \cite[p. xi]{GGG}, all theorems of the present article can be transferred to   totally geodesic Radon transforms
on the hyperbolic and elliptic spaces. Almost all  formulas, which are needed for this transition, are
available in the literature. Specifically,
 for the hyperplane Radon transforms and the k-plane transform see
 \cite{BCK, BR99, Ku}. The correspondence between the affine $j$-plane to $k$-plane transform
 and the similar transform for planes through the origin was established in \cite{Ru04d}.
 We believe that a similar transition  holds to the  hyperbolic space.

{\bf 2.} It might be  challenging to establish connection between weighted $L^p$ inequalities of the present article and
   known $L^p$-$L^q$
or mixed norm estimates.

{\bf 3.} The $k$-plane transform (\ref{uu7}) is a member of the analytic family of the Semyanistyi type integrals
\[ (R_k^\a
f)(\t)\!=\!\frac{1}{\gamma_{n-k}(\a)}\irn \!\!f(x) |x\!-\!\t|^{\a+k-n}  dx, \quad \gamma_{n-k, \a}\!=\!
  \frac{2^\a\pi^{(n-k)/2}\Gamma(\a/2)}{\Gamma((n\!-\!k\!-\!\a)/2)},
\]
where $Re \, \a >0$ and $|x-\t|$ denotes the Euclidean distance between $x\in \rn$ and
$\t\in \cgnk$; see \cite{Se} for $k=n-1$ and \cite[p. 104]{Ru04b} for any $1\le k<n$. We conjecture that the method of our article extends
to these operators and their duals. Since $\lim\limits_{\a \to 0} R_k^\a f=R_k f$ in a suitable sense, we expect that
$\lim\limits_{\a \to 0}||R_k^\a||=||R_k||$ on the relevant weighted spaces.

{\bf 4.} To the best of our knowledge, Semyanistyi type integrals associated to the $j$-plane to $k$-plane transform $R_{j,k}$ are unknown. It might be interesting to properly introduce them and study their mapping properties; see \cite{OR} for similar operators on matrix spaces.

\bibliographystyle{amsplain}

\end{document}